\documentclass{amsart}
\usepackage[margin=3.5cm]{geometry}
\usepackage[dvipsnames]{xcolor}

\usepackage{amsfonts, amsthm, amssymb,verbatim,amscd,amsmath, blindtext}

\usepackage{multirow,multicol}
\usepackage{graphicx}
\graphicspath{ {./images/} }
\usepackage{enumitem,yhmath}
\usepackage{amssymb}
\usepackage{ytableau}
\usepackage{tikz}
\usepackage{tikz-cd,rotating}
\usepackage{float, pifont}
\usepackage[utf8]{inputenc}
\usepackage{mathtools,nccmath}
\usepackage{float}
\usepackage{adjustbox}
\usepackage[colorlinks=true,citecolor=blue,linkcolor=BrickRed]{hyperref}
\usepackage{caption} 
\newtheorem{theorem}{Theorem}[section]
\newtheorem{lemma}[theorem]{Lemma}

\theoremstyle{remark}
\newtheorem{remark}[theorem]{Remark}
\newtheorem{question}[theorem]{Question}

\theoremstyle{definition}
\newtheorem{definition}[theorem]{Definition}

\DeclareMathOperator{\Min}{Min}

\title{Associated primes and v-numbers of integral closures of powers of edge ideals}
\author[Chau]{Trung Chau}
\address{Department of Mathematics, University of Manitoba, 420 Machray Hall, 186 Dysart
Road, Winnipeg, MB R3T 2N2 Canada}
\email{chauchitrung1996@gmail.com}

\author{T\`ai Huy H\`a}
\address{Tulane University, Mathematics Department, 6823 St. Charles Avenue, New Orleans, LA 70118, USA}
\email{tha@tulane.edu}

\author{Thanh Vu}
\address{Institute of Mathematics, VAST, 18 Hoang Quoc Viet, Hanoi, Vietnam}
\email{vuqthanh@gmail.com}

\begin{document}

\begin{abstract}
Let $I = I(G)$ be the edge ideal of a finite simple graph $G$. We investigate the relationship between the associated primes and $v$-numbers of powers $I^t$ and their integral closures $\overline{I^t}$. For the third power, we establish the associated prime containment $\operatorname{Ass}(S/\overline{I^3}) \subseteq \operatorname{Ass}(S/I^3)$, which yields $v(I^3) \le v(\overline{I^3})$. In contrast, for $t \ge 4$, we show that $\operatorname{Ass}(S/\overline{I^t})$ and $\operatorname{Ass}(S/I^t)$ are in general incomparable. Furthermore, for every fixed $t \ge 3$ and $r\ge 4$, we construct families of graphs for which the discrepancies $v(\overline{I^t}) - v(I^t)$ and $v(I^r)-v(\overline{I^r})$ are arbitrarily large. 
\end{abstract}

\maketitle

\section{Introduction}
Let $S = \mathrm{k}[x_1, \ldots, x_n]$ be a standard graded polynomial ring over a field $\mathrm{k}$, and let $I \subseteq S$ be a nonzero homogeneous ideal. For an associated prime $P \in \operatorname{Ass}(S/I)$, the \emph{local $v$-number} of $I$ at $P$ is defined by
\[
v_P(I) = \min \{ d \ge 0 \mid \text{there exists } f \in S_d \text{ such that } (I : f) = P \}.
\]
The \emph{$v$-number} of $I$, introduced by Cooper et al.~\cite{CSTVV} in connection with the minimum distance function of projective Reed--Muller-type codes, is given by
\[
v(I) = \min_{P \in \operatorname{Ass}(S/I)} v_P(I).
\]
The $v$-number of edge ideals was studied by Jaramillo and
Villarreal~\cite{JV}. Conca~\cite{Conca} and Ficarra and
Sgroi~\cite{FS} established the eventual linearity of the
$v$-numbers of ordinary powers of homogeneous ideals. More
recently, Ficarra and Sgroi~\cite{FSFil} studied the
$v$-function of Noetherian graded filtrations and proved its
eventual quasi-linearity, including for integral closure
filtrations.

Vanmathi A and Sarkar~\cite{VS} investigated the $v$-numbers
of integral closures of powers of monomial ideals. Biswas,
Mandal, and Phukan~\cite{BMP} showed that the $v$-numbers
of a monomial ideal and its integral closure need not be
comparable. They also proved that if $G$ is a connected graph,
then $v(I(G)^t)=v(\overline{I(G)^t})$ for
$t\ge 1+|E(G)|$, and obtained an eventual comparison for
disconnected graphs. In this paper, we investigate the
relationship between $v(I(G)^t)$ and $v(\overline{I(G)^t})$
at individual powers. Since $I(G)^t=\overline{I(G)^t}$ for
$t=1,2$, the third power is the first instance where
the two ideals, and consequently their associated primes and
$v$-numbers, can diverge.

On the associated prime side, Trung~\cite{Trung09} studied
the stability of the associated primes of integral closures
of powers of monomial ideals. For edge ideals, Mart\'inez-Bernal,
Morey, and Villarreal~\cite{MMV} showed that the eventual
stable sets of associated primes of ordinary powers and their
integral closures coincide, while Mau and Trung~\cite{MT}
obtained bounds for the stability indices of associated primes
of integral closures. In contrast to the combinatorial
description of associated primes of ordinary powers given by
Lam and Trung~\cite{LT}, a general description of
$\operatorname{Ass}(S/\overline{I(G)^t})$ for arbitrary graphs
and powers is not available.

Because the $v$-number is defined in terms of associated
primes, we first investigate their relationship for ordinary
powers and their integral closures. We also compare the first
powers at which the homogeneous maximal ideal becomes an
associated prime. Our main contributions are organized as
follows:

\begin{enumerate}
    \item \textbf{Local equality at minimal primes.}
    For every minimal prime $P \in \operatorname{Min}(S/I(G))$ and every $t \ge 1$, we prove that the local $v$-numbers coincide:
    \[
    v_P(I(G)^t) = v_P(\overline{I(G)^t}).
    \]

    \item \textbf{Containment and inequality for the third power.}
    For $t = 3$, we prove that 
    \[
    \operatorname{Ass}(S/\overline{I(G)^3}) \subseteq \operatorname{Ass}(S/I(G)^3) \text{ and }
    v(I(G)^3) \le v(\overline{I(G)^3}).
    \]

    \item \textbf{Incomparability for higher powers.}
    For every fixed integer $t \ge 4$, we show that $\operatorname{Ass}(S/\overline{I(G)^t})$ and $\operatorname{Ass}(S/I(G)^t)$ are, in general, incomparable under set inclusion. Moreover, the first power at which the homogeneous maximal ideal becomes associated to an integral closure can be arbitrarily smaller than the corresponding first power for ordinary powers.

    \item \textbf{Arbitrarily large gaps for higher powers.}
    For every fixed integer $t \ge 3$ and any integer $r \ge 1$, we construct families of graphs for which the discrepancy
    \[
    v(\overline{I(G)^t}) - v(I(G)^t) \ge r - 1.
    \]
    Furthermore, for $t \ge 4$, we show that the reverse difference $v(I(G)^t) - v(\overline{I(G)^t})$ can also be made arbitrarily large.
\end{enumerate}

\medskip

\noindent\textbf{Acknowledgment.} Chau gratefully acknowledges that this research was supported in part by the Pacific Institute for the Mathematical Sciences and a grant from the Simons Foundation International [SFI-MPS-T-Institutes-00020822, OY]. H\`a is partially supported by a Simons Foundation~grant. u is partially supported by the National Foundation for Science and Technology Development under grant number Nafosted/101/10104/2026/1279.

\noindent\textbf{AI Disclosure.} Theorem~\ref{thm:ordinary-larger-arbitrarily} was proved by ChatGPT-6.1 Sol in the case $t=4$ and generalized by the authors. 

\section{Preliminaries}

We collect the graph-theoretic notation and the criterion for membership in the integral closure that will be used throughout the paper.

\begin{definition}
    
Let $G$ be a simple graph with vertex set $V(G) = \{1, \ldots, n\}$ and edge set $E(G)$.
\begin{enumerate}
    \item A simple graph $H$ is a \emph{subgraph} of $G$ if $V(H) \subseteq V(G)$ and $E(H) \subseteq E(G)$. It is an \emph{induced subgraph} of $G$ if $E(H) = \{\{u,v\} \in E(G) \mid u, v \in V(H)\}$.
    
    \item For a subset $U \subseteq V(G)$, we denote by $G[U]$ and $G - U$ the induced subgraphs of $G$ on $U$ and on $V(G) \setminus U$, respectively.
    
    \item For a vertex $v \in V(G)$, $N_G(v)$ denotes the set of neighbors of $v$. For a subset $U \subseteq V(G)$, $N_G(U) = \bigcup_{u \in U} N_G(u)$ denotes the neighborhood of $U$, and $N_G[U] = N_G(U) \cup U$ denotes the closed neighborhood of $U$.

   \item A subset $X \subseteq V(G)$ is called an \emph{independent set} if no two vertices in $X$ are adjacent in $G$. We denote by $\alpha(G)$ the maximum size of an independent set of $G$.
   
    \item A subset $C \subseteq V(G)$ is called a \emph{vertex cover} (or simply a \emph{cover}) of $G$ if it meets every edge of $G$; that is, $C \cap e \neq \emptyset$ for all $e \in E(G)$.

    \item A graph $G$ is called \emph{bipartite} if its vertex set can be partitioned as $V(G) = U \cup V$ with $U \cap V = \emptyset$ such that every edge in $E(G)$ has one endpoint in $U$ and the other in $V$.
\item A graph $G$ is called \emph{strongly non-bipartite} if each of its connected components is non-bipartite.

\item An edge $e$ of $G$ is called a bridge if its removal increases the number of connected components by one.

    \item A \emph{matching} in $G$ is a set of pairwise non-adjacent edges. The \emph{matching number} of $G$, denoted by $\operatorname{mat}(G)$, is the size of a maximum matching of $G$. A \emph{perfect matching} is a matching that covers every vertex of $G$.
    
    \item A graph $G$ is called \emph{factor-critical} if $G - v$ has a perfect matching for every $v \in V(G)$.
    
\end{enumerate}
\end{definition}
Throughout this paper, we denote by $S = \mathrm{k}[x_1, \ldots, x_n]$ the polynomial ring over a field $\mathrm{k}$, and by $\mathfrak{m} = (x_1, \ldots, x_n)$ its maximal homogeneous ideal. For a nonzero monomial $f \in S$, its \emph{support}, denoted by $\operatorname{supp}(f)$, is the set of all variables in $S$ that divide $f$. For each $i \in [n]$, we denote by $\deg_i(f)$ the exponent of $x_i$ in $f$, and by $\deg(f)$ the total degree of $f$. For any subset $X \subseteq [n]$, we define
\[
\deg_X(f) = \sum_{i \in X} \deg_i(f).
\]

\begin{definition}
    
Let $G$ be a simple graph on $V(G) = [n]$ and let $\mathbf{a} = (a_1, \ldots, a_n) \in \mathbb{N}^n$ be a vector of positive integers. The \emph{parallelization} of $G$ with respect to $\mathbf{a}$, denoted by $G_{\mathbf{a}}$, is the graph obtained by replicating each vertex $i \in V(G)$ into $a_i$ copies, denoted $i_1, \ldots, i_{a_i}$, such that $\{i_k, j_\ell\} \in E(G_{\mathbf{a}})$ if and only if $\{i,j\} \in E(G)$.
\end{definition}

For the subsequent arguments, we use the fractional matching number with capacity and its dual formulation.
\begin{definition}
Let $G$ be a simple graph with $V(G) = \{v_1, \ldots, v_n\}$, and let $\mathbf{a} \in \mathbb{N}^n$ be an exponent vector. The \emph{fractional matching number of $G$ with capacity $\mathbf{a}$} is defined as
\[
\nu_G^*(\mathbf{a}) = \max \left\{ \sum_{e \in E(G)} z_e \;\middle|\; z_e \ge 0 \text{ for all } e \in E(G), \text{ and } \sum_{e \ni v} z_e \le a_v \text{ for all } v \in V(G) \right\}.
\]
\end{definition}

We have the following membership criterion in terms of $\nu_G^*(\mathbf{a})$, see \cite[Proposition 1.1]{HT}.

\begin{lemma}\label{lem_criterion_int}
Let $G$ be a simple graph and let $\mathbf{a} \in \mathbb{N}^n$. Then $x^{\mathbf{a}} \in \overline{I(G)^t}$ if and only if $\nu_G^*(\mathbf{a}) \ge t$.
\end{lemma}

Let $Q$ be the covering polyhedron of $G$, namely 
\[
Q = \left\{ \mathbf{y} \in \mathbb{R}^{|V(G)|}_{\ge 0} \;\middle|\; y_u + y_v \ge 1 \text{ for every edge } \{u,v\} \in E(G) \right\}.
\]
By linear programming duality, $\nu_G^*(\mathbf{a})$ can also be expressed as the optimal value of the dual problem:
\[
\nu_G^*(\mathbf{a}) = \min \left\{ \sum_{v \in V(G)} a_v y_v \;\middle|\; \mathbf{y} \in Q \right\}.
\]

\section{Associated primes of ordinary powers and their integral closures}

In this section, we investigate associated primes of powers of an edge ideal and their integral closures. Our first result establishes a containment of associated primes for the third power.

\begin{theorem}\label{thm:third_power}
Let $I=I(G)$ be the edge ideal of a simple graph $G$. If $f$ is a monomial such that $(\overline{I^3}:f)=P$ is a prime ideal, then $(I^3:f)=P$. In particular,
\[
\operatorname{Ass}(S/\overline{I^3})\subseteq\operatorname{Ass}(S/I^3)
\quad\text{and}\quad
v(I^3)\le v(\overline{I^3}).
\]
\end{theorem}

\begin{proof}
By \cite[Theorem 2.10]{MV},
\[
\overline{I^3}=I^3+(m_Cm_D\mid C,D\text{ are exceptional pairs of triangles of }G),
\]
where $m_C=\prod_{v\in C}x_v$, and $C,D$ form an exceptional pair if $N_G[C]\cap D=\emptyset$. Set $L=\overline{I^3}$, and suppose that $(L:f)=P$ is a monomial prime ideal. Since $I^3\subseteq L$, we have $(I^3:f)\subseteq P$.

Assume for contradiction that there exists a variable $x\in P$ such that $xf\notin I^3$. Then $xf\in L\setminus I^3$, so $xf$ is divisible by $m_Cm_D$ for some exceptional pair of triangles $C,D$. Write $C=\{x,y,z\}$ and $D=\{a,b,c\}$. Since $f\notin L$, the square-free monomial $m_Cm_D$ cannot divide $f$. Thus $x\nmid f$ and $yzabc\mid f$. Furthermore, $y^2\nmid f$ and $z^2\nmid f$: otherwise $xy,yz,ab$ or $xz,yz,ab$, respectively, would be three edges whose product divides $xf$.

We claim that $yf\notin I^3$. Otherwise, since $f\notin I^3$ and $y$ occurs to exponent one in $f$, any product of three edges dividing $yf$ must use both copies of $y$. At most one of these edges can be $yz$, since $z$ also occurs to exponent one in $f$. As $x\nmid f$ and no vertex of $D$ is adjacent to $y$, at least one of the two edges incident to $y$ must be $yu$ for some $u\notin C\cup D$. But then the product of the three edges $yu,xz,ab$ divides $xf$, a contradiction. By symmetry, $zf\notin I^3$.

The monomials $yf$ and $zf$ have the same support as $f$. Every generator $m_Cm_D$ in the above description of $L$ is square-free. Therefore, if an exceptional generator divided $yf$ or $zf$, it would also divide $f$, which is impossible. We conclude that $yf,zf\notin L$, and hence $y,z\notin P$. But $P$ contains $I$, so it is a vertex cover prime of $G$. This contradicts $yz\in I$.

Thus $P\subseteq(I^3:f)$, and equality follows. Finally, choose $f$ with $(L:f)=P$ and $\deg(f)=v_P(L)$. We obtain $v_P(I^3)\le v_P(L)$ for every $P\in\operatorname{Ass}(S/L)$, which gives the asserted containment and inequality.
\end{proof}

We next give a criterion for the homogeneous maximal ideal to be an associated prime of the integral closure of a power. We first show that every monomial witness for the maximal ideal has a fixed degree.

\begin{lemma}\label{lem_deg_maximal_colon}
Let $G$ be a simple graph and $I=I(G)$. Suppose that $f=x^{\mathbf{a}}$ is a monomial such that $(\overline{I^t}:f)=\mathfrak{m}$. Then $\deg(f)=2t-1$.
\end{lemma}

\begin{proof}
By Lemma \ref{lem_criterion_int},
\[
\nu_G^*(\mathbf{a})<t
\quad\text{and}\quad
\nu_G^*(\mathbf{a}+\mathbf{e}_v)\ge t
\quad\text{for every }v\in V(G).
\]
Let $\mathbf{y}$ be an optimal extreme point of the covering polyhedron $Q$ for $\nu_G^*(\mathbf{a})$. We claim that $y_v>0$ for all $v\in V(G)$. Indeed, if $y_v=0$, then
\[
\nu_G^*(\mathbf{a}+\mathbf{e}_v)
\le (\mathbf{a}+\mathbf{e}_v)\cdot\mathbf{y}
=\mathbf{a}\cdot\mathbf{y}
=\nu_G^*(\mathbf{a})<t,
\]
a contradiction.

Let $H$ be the spanning subgraph of $G$ whose edges are those $\{u,v\}$ satisfying $y_u+y_v=1$. Every connected component of $H$ must contain an odd cycle. Otherwise, a bipartite component, including the possibility of an isolated vertex, would allow us to perturb $\mathbf{y}$ by small positive and negative multiples of the vector taking the values $1$ and $-1$ on its two parts and zero elsewhere. Since all coordinates of $\mathbf{y}$ are positive, the nonnegativity inequalities remain satisfied; the tight edge inequalities are unchanged, and the remaining inequalities remain satisfied for sufficiently small perturbations. This contradicts the fact that $\mathbf{y}$ is an extreme point of $Q$.

On each connected component of $H$, the tight equalities, together with an odd cycle, force $y_v=1/2$ at every vertex. Hence $\mathbf{y}=(1/2,\ldots,1/2)$, and therefore
\[
\nu_G^*(\mathbf{a})=\frac{\deg(f)}2<t.
\]
On the other hand, for any vertex $v$, the same feasible dual vector gives
\[
t\le\nu_G^*(\mathbf{a}+\mathbf{e}_v)
\le\frac{\deg(f)+1}{2}.
\]
Since $\deg(f)$ is an integer, these inequalities imply $\deg(f)=2t-1$.
\end{proof}

The following criterion strengthens the necessary condition in Lemma \ref{lem_deg_maximal_colon} and its converse. For a subset $U\subseteq V(G)$, we write $\mathbf{a}(U)=\sum_{u\in U}a_u$.

\begin{lemma}\label{lem_capacity}\label{lem:weighted_hall}
Let $G$ be a simple graph, let $I=I(G)$, and let $f=x^{\mathbf{a}}$ be a monomial. Then $(\overline{I^t}:f)=\mathfrak{m}$ if and only if $\deg(f)=2t-1$ and
\[
\mathbf{a}(N_G(U))\ge\mathbf{a}(U)+1
\]
for every nonempty independent set $U\subseteq V(G)$.
\end{lemma}

\begin{proof}
Suppose that $(\overline{I^t}:f)=\mathfrak{m}$. By Lemma \ref{lem_deg_maximal_colon}, $\deg(f)=2t-1$. Let $U$ be a nonempty independent set and choose $v\in U$. Since $x_vf\in\overline{I^t}$, there exists a fractional matching with capacity $\mathbf{b}=\mathbf{a}+\mathbf{e}_v$ and total weight at least $t$. As $\sum_{u\in V(G)}b_u=2t$, its total weight must be exactly $t$, and all vertex capacities are saturated. Since no edge has both endpoints in $U$, all capacity used at vertices of $U$ must be matched with vertices of $N_G(U)$. Thus
\[
\mathbf{a}(N_G(U))=\mathbf{b}(N_G(U))
\ge\mathbf{b}(U)=\mathbf{a}(U)+1.
\]

Conversely, assume that $\deg(f)=2t-1$ and that the stated inequalities hold. In particular, $f\notin\overline{I^t}$ by degree. Fix $v\in V(G)$ and put $\mathbf{b}=\mathbf{a}+\mathbf{e}_v$. By Lemma \ref{lem_criterion_int}, it is enough to show that $\nu_G^*(\mathbf{b})\ge t$.

The assumed inequalities imply that $G$ has no isolated vertices.
Let $\mathbf{y}$ be an extreme point of $Q$. By
\cite[Lemma~4.5]{DHNT}, every coordinate of $\mathbf{y}$ belongs to
$\{0,1/2,1\}$. Set
\[
Z=\{u:y_u=0\},\quad M=\{u:y_u=1/2\},\quad U=\{u:y_u=1\}.
\]
The set $Z$ is independent and $N_G(Z)\subseteq U$. Since $\deg(f)=2t-1$, we have
\[
\mathbf{b}\cdot\mathbf{y}
=\frac{2t-1+\mathbf{a}(U)-\mathbf{a}(Z)}2+y_v.
\]
If $Z=\emptyset$, then $y_v\ge1/2$ and $\mathbf{b}\cdot\mathbf{y}\ge t$. If $Z\ne\emptyset$, the hypothesis gives
$\mathbf{a}(U)\ge\mathbf{a}(N_G(Z))\ge\mathbf{a}(Z)+1$, and again $\mathbf{b}\cdot\mathbf{y}\ge t$. Thus the dual objective is at least $t$ at every extreme point of $Q$. By linear programming duality, $\nu_G^*(\mathbf{b})\ge t$.

It follows that $x_vf\in\overline{I^t}$ for every vertex $v$. Since $f\notin\overline{I^t}$, we have $(\overline{I^t}:f)=\mathfrak{m}$, as desired.
\end{proof}

As an immediate consequence, we recover the following sufficient condition, which will also be used in the next section.

\begin{lemma}\label{lem:stricthall}
Let $G$ be a graph on an odd number $N$ of vertices and assume that
\[
|N_G(U)|>|U|
\]
for every nonempty independent set $U\subseteq V(G)$. Let $I=I(G)$, let $\mathfrak{m}$ be the homogeneous maximal ideal of $S$, and set $t=(N+1)/2$. Then
\[
\mathfrak{m}\in\operatorname{Ass}(S/\overline{I^t}).
\]
\end{lemma}

\begin{proof}
Take $f=\prod_{v\in V(G)}x_v$. Then $\deg(f)=N=2t-1$, and the inequalities in Lemma \ref{lem_capacity} are precisely the assumed strict inequalities. Thus $(\overline{I^t}:f)=\mathfrak{m}$.
\end{proof}

We recall the generalized ear decompositions and invariants of Lam and Trung \cite{LT}, which will be used to compare the associated primes of ordinary powers.

\begin{definition}
A \emph{generalized ear decomposition} of a graph $G$ is a sequence $\mathcal{E}$ of walks in $G$ such that the first walk is closed, the endpoints of each subsequent walk belong to earlier walks, and every vertex of $G$ appears exactly once as an inner point of $\mathcal{E}$. A vertex is an inner point if it is neither the last vertex of the first walk nor an endpoint of any subsequent walk; we count the first vertex of the first walk as an inner point. A generalized ear decomposition is \emph{odd-beginning} if the first walk in each connected component is an odd cycle.
\end{definition}

\begin{definition}
For a strongly non-bipartite graph $G$, let $\varphi^*(G)$ be the minimum number of even walks in odd-beginning generalized ear decompositions of its connected components, and set
\[
\mu^*(G)=\frac{\varphi^*(G)+|V(G)|-c(G)}2,
\]
where $c(G)$ denotes the number of connected components of $G$.
\end{definition}

We will use the criterion \cite[Theorem 3.6]{LT}: $\mathfrak{m}\in\operatorname{Ass}(S/I(G)^t)$ if and only if there exists a dominating set $U\subseteq V(G)$ such that $G[U]$ is strongly non-bipartite and $\mu^*(G[U])<t$. We will also use \cite[Lemma 3.9]{LT}, which gives $\mu^*(H)=\mu^*(H-w)+1$ when $H$ is strongly non-bipartite and $w$ is a leaf.

\begin{lemma}\label{lem_bridge_edge}
Let $G$ be a connected graph and let $e$ be a bridge of $G$. Then every generalized ear decomposition of $G$ contains the closed walk of length two along $e$.
\end{lemma}

\begin{proof}
Let $A$ and $B$ be the components of $G-e$, and suppose that the first walk of a generalized ear decomposition lies in $A$. This is automatic unless the first walk is the closed walk of length two along $e$, in which case there is nothing to prove. Let $W$ be the first walk using $e$. Before $W$, all vertices that have appeared lie in $A$, so both endpoints of $W$ lie in $A$. The walk $W$ must cross $e$ into $B$ and return along $e$. Since its inner vertices are new and distinct, it can contain no additional vertices between the two crossings, and it must be the closed walk of length two along $e$. The same argument applies if the first walk lies in $B$.
\end{proof}

We can now construct examples for which an integral closure has the maximal ideal as an associated prime while the corresponding ordinary power does not.

\begin{theorem}\label{thm:noncontainment}
For every integer $t\ge4$, there exists a simple graph $G$ such that
\[
\mathfrak{m}\in\operatorname{Ass}(S/\overline{I(G)^t})
\quad \text{but} \quad
\mathfrak{m}\notin\operatorname{Ass}(S/I(G)^t).
\]
\end{theorem}

\begin{proof}
Fix $t\ge4$ and set $k=t-3$. Let $C_3$ and $C_{2k+1}$ be disjoint odd cycles, and let $z$ be a vertex joined to one vertex of each cycle. Denote by $B$ the graph on $C_3\cup C_{2k+1}\cup\{z\}$. Let $G$ be obtained from $B$ by attaching one leaf to every vertex of $B$. Thus $|V(B)|=2t-1$. Consider the monomial
\[
f=\prod_{v\in V(B)}x_v.
\]
Since $\deg(f)=2t-1$, we have $f\notin\overline{I(G)^t}$. We claim that $x_wf\in\overline{I(G)^t}$ for every vertex $w$ of $G$.

If $w$ belongs to one of the odd cycles, we may use the two edges of that cycle incident to $w$, together with a perfect matching of the remaining vertices of that cycle. The other odd cycle, together with $z$, has a perfect matching. These edges give $t$ edge factors dividing $x_wf$. If $w$ is the leaf attached to a vertex $v$ of either cycle, then $B-v$ has a perfect matching of size $t-1$. Together with the edge $wv$, this gives $t$ edge factors dividing $x_wf$.

If $w=z$, we use the two edges joining $z$ to the two cycles and perfect matchings of the even paths obtained by deleting their neighbors from the cycles. Again, the product of these $t$ edges divides $x_zf$. Finally, if $w$ is the leaf attached to $z$, assign fractional matching weight $1$ to the edge $wz$ and weight $1/2$ to every edge of the two odd cycles. This is a fractional matching with capacity given by the exponent vector of $x_wf$ and total weight
\[
1+\frac32+\frac{2k+1}{2}=k+3=t.
\]
By Lemma \ref{lem_criterion_int}, $x_wf\in\overline{I(G)^t}$ in this last case as well. Hence $(\overline{I(G)^t}:f)=\mathfrak{m}$.

It remains to prove that $\mathfrak{m}\notin\operatorname{Ass}(S/I(G)^t)$. Let $U$ be a dominating set of $G$ such that $G[U]$ is strongly non-bipartite. We claim that $V(B)\subseteq U$. Indeed, if $v\in V(B)\setminus U$, its attached leaf must belong to $U$ in order to be dominated. That leaf would then be an isolated vertex of $G[U]$, a contradiction.

Thus $G[U]$ is obtained from $B$ by adding some of its leaves. By \cite[Lemma 3.9]{LT}, $\mu^*(G[U])\ge\mu^*(B)$. The two edges joining $z$ to the cycles are bridges. By Lemma \ref{lem_bridge_edge}, every generalized ear decomposition of $B$ has at least two even walks. Conversely, begin with one odd cycle, add the two bridge edges as closed walks of length two, and then add the other odd cycle. This gives an odd-beginning generalized ear decomposition with exactly two even walks. Therefore $\varphi^*(B)=2$, and
\[
\mu^*(B)=\frac{2+(2t-1)-1}{2}=t.
\]
Consequently, $\mu^*(G[U])\ge t$ for every such dominating set $U$. By \cite[Theorem 3.6]{LT}, $\mathfrak{m}\notin\operatorname{Ass}(S/I(G)^t)$.
\end{proof}

We now compare the first powers for which the maximal ideal becomes an associated prime. Set
\[
\iota_{\rm ord}(G)=\min\{t\ge1\mid\mathfrak{m}\in\operatorname{Ass}(S/I(G)^t)\}
\quad\text{and}\quad
\iota_{\rm int}(G)=\min\{t\ge1\mid\mathfrak{m}\in\operatorname{Ass}(S/\overline{I(G)^t})\}.
\]
For an integer $L\ge1$, let $G_L$ be the graph obtained from the path
\[
v_0-v_1-\cdots-v_L
\]
by completing $v_0$ and $v_L$ to two triangles. Thus we add vertices $a_0,b_0,a_L,b_L$ and the six edges
\[
v_0a_0,\ v_0b_0,\ a_0b_0,\quad
v_La_L,\ v_Lb_L,\ a_Lb_L.
\]
The graph has $L+5$ vertices.

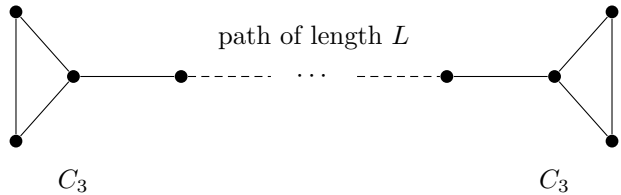
\begin{figure}[ht]
\centering
\begin{tikzpicture}[scale=.95, every node/.style={circle,fill=black,inner sep=1.7pt}]
\node (v0) at (0,0) {};
\node (a0) at (-.8,.9) {};
\node (b0) at (-.8,-.9) {};
\draw (v0)--(a0)--(b0)--(v0);
\node (v1) at (1.5,0) {};
\node[fill=none,inner sep=0pt] (dots) at (3.35,0) {$\cdots$};
\node (vL1) at (5.2,0) {};
\node (vL) at (6.7,0) {};
\draw (v0)--(v1);
\draw[densely dashed] (v1)--(2.75,0);
\draw[densely dashed] (3.95,0)--(vL1);
\draw (vL1)--(vL);
\node (aL) at (7.5,.9) {};
\node (bL) at (7.5,-.9) {};
\draw (vL)--(aL)--(bL)--(vL);
\node[fill=none] at (0,-1.45) {$C_3$};
\node[fill=none] at (6.7,-1.45) {$C_3$};
\node[fill=none] at (3.35,.55) {path of length $L$};
\end{tikzpicture}
\caption{The graph $G_L$: two triangles joined by a path of length $L$.}
\label{fig:GL}
\end{figure}

We first verify the strict Hall condition on $G_L$.

\begin{lemma}\label{lem:GLhall}
For every $L\ge1$ and every nonempty independent set $U\subseteq V(G_L)$,
\[
|N_{G_L}(U)|>|U|.
\]
Moreover, if $v_0\in U$, then $|N_{G_L}(U)|\ge|U|+2$.
\end{lemma}

\begin{proof}
Let $P=v_0-v_1-\cdots-v_L$, and put $A=U\cap V(P)$. An independent set $A=\{v_{i_1},\ldots,v_{i_q}\}$ of $P$, with $i_1<\cdots<i_q$, has at least $q-1$ neighbors in $P$, by choosing one vertex strictly between each two consecutive vertices of $A$. If $A$ avoids both endpoints, it has at least $q+1$ neighbors in $P$. If $A$ contains exactly one endpoint, it has at least $q$ neighbors in $P$.

Let $e$ be the number of vertices of $U$ among $\{a_0,b_0,a_L,b_L\}$ and let $r=|A\cap\{v_0,v_L\}|$. Since $U$ is independent, at most one of $a_0,b_0$ and at most one of $a_L,b_L$ belongs to $U$, and none of these outer vertices belongs to $U$ when the corresponding path endpoint is in $U$.

If $A\ne\emptyset$ and $r=0$, the selected outer vertices contribute $e$ distinct neighbors outside $P$, so
\[
|N_{G_L}(U)|\ge|N_P(A)|+e\ge|A|+e+1=|U|+1.
\]
If $r\ge1$, each selected path endpoint contributes two neighbors outside $P$, and each selected outer vertex contributes its partner in its triangle. Thus
\[
|N_{G_L}(U)|\ge|N_P(A)|+2r+e\ge|A|-1+2r+e\ge|U|+1.
\]
If $r=1$, the stronger bound $|N_P(A)|\ge|A|$ gives $|N_{G_L}(U)|\ge|U|+2$. If $r=2$, the preceding bound already gives $|N_{G_L}(U)|\ge|U|+3$. This proves the additional assertion when $v_0\in U$. Finally, if $A=\emptyset$, the selected outer vertices have disjoint neighborhoods of size two, giving $|N_{G_L}(U)|\ge2|U|>|U|$.
\end{proof}

\begin{lemma}\label{lem_phi_GL}
For every $L\ge3$, $\varphi^*(G_L)=L$.
\end{lemma}

\begin{proof}
Every edge of the connecting path is a bridge, so Lemma \ref{lem_bridge_edge} shows that an odd-beginning generalized ear decomposition of $G_L$ has at least $L$ even walks. Conversely, begin with the triangle $v_0,a_0,b_0,v_0$, add the closed walks $v_{i-1},v_i,v_{i-1}$ for $i=1,\ldots,L$, and finish with the triangle $v_L,a_L,b_L,v_L$. This is an odd-beginning generalized ear decomposition with exactly $L$ even walks.
\end{proof}

\begin{lemma}\label{lem:sGL}
For every $L\ge3$, $\iota_{\rm ord}(G_L)=L$.
\end{lemma}

\begin{proof}
By \cite[Theorem 3.6]{LT}, it suffices to prove that
\[
\min\{\mu^*(G_L[U])\mid U\subseteq V(G_L),\ U\text{ is dominating},\ G_L[U]\text{ is strongly non-bipartite}\}=L-1.
\]
Let $U$ be such a dominating set and put $H=G_L[U]$. Since the only odd cycles of $G_L$ are its two endpoint triangles, $H$ has at most two connected components. By repeated application of \cite[Lemma 3.9]{LT}, a triangle with an attached path of length $p$ has
\begin{equation}\label{eq_triangle}
\mu^*=p+1.
\end{equation}

If $H$ has two connected components, both endpoint triangles are contained in $U$. The vertices of $U$ on the connecting path form a prefix $v_0,\ldots,v_i$ and a suffix $v_j,\ldots,v_L$. The omitted middle block has length $r=j-i-1$, where $1\le r\le2$ by domination. Hence
\[
\mu^*(H)=(i+1)+(L-j+1)=L-r+1\ge L-1.
\]
The value $L-1$ is attained by omitting two consecutive internal vertices of the connecting path.

If $H$ is connected and contains only one endpoint triangle, it must contain the entire connecting path and possibly one further vertex of the opposite endpoint triangle. It is therefore a triangle with a path of length $L$ or $L+1$ attached, so $\mu^*(H)\ge L+1$. Finally, if $H$ is connected and contains both endpoint triangles, then $H=G_L$, and Lemma \ref{lem_phi_GL} gives
\[
\mu^*(H)=\frac{L+(L+5)-1}{2}=L+2.
\]
Taking the minimum gives the claimed formula.
\end{proof}

The integral closure first appearance index can also be determined exactly.

\begin{lemma}\label{lem:int_GL}
For every integer $L\ge3$,
\[
\iota_{\rm int}(G_L)=\left\lceil\frac{L+6}{2}\right\rceil.
\]
\end{lemma}

\begin{proof}
First suppose that $L$ is even. By Lemmas \ref{lem:GLhall} and \ref{lem:stricthall}, the product of all $L+5$ variables is a witness for the maximal ideal of $\overline{I(G_L)^{(L+6)/2}}$. If $L$ is odd, take the same product multiplied by $x_{v_0}$. The additional assertion of Lemma \ref{lem:GLhall} shows that its exponent vector satisfies the inequalities of Lemma \ref{lem_capacity}, and its total degree is $L+6$. It is therefore a witness for $\overline{I(G_L)^{(L+7)/2}}$. In both cases we obtain
\[
\iota_{\rm int}(G_L)\le\left\lceil\frac{L+6}{2}\right\rceil.
\]

To prove the reverse inequality, let $f=x^{\mathbf{a}}$ be a monomial such that $(\overline{I(G_L)^s}:f)=\mathfrak{m}$, and set $U=\operatorname{supp}(f)$. By Lemma \ref{lem_capacity}, $U$ is a dominating set. Moreover, each connected component of $G_L[U]$ is non-bipartite. Indeed, if such a component had bipartition $A\cup B$, both parts would be independent, and all neighbors of $A$ and $B$ outside this component would have exponent zero. The inequalities in Lemma \ref{lem_capacity} would therefore give $\mathbf{a}(B)\ge\mathbf{a}(A)+1$ and $\mathbf{a}(A)\ge\mathbf{a}(B)+1$, a contradiction. The same argument with a singleton rules out an isolated vertex.

We need a bound for the weight of a component $W$ of $G_L[U]$ consisting of a triangle with an attached path $u_0,u_1,\ldots,u_p$, where $u_0$ is a vertex of the triangle. We claim that
\[
\mathbf{a}(W)\ge2p+3.
\]
This is clear if $p=0$. Suppose $p\ge1$. For $0\le i\le p-1$, the set $\{u_{i+1},u_{i+3},\ldots\}\cap\{u_0,\ldots,u_p\}$ is independent. Its neighbors of positive exponent are precisely the alternate vertices $u_i,u_{i+2},\ldots$. Thus Lemma \ref{lem_capacity} gives
\[
b_i:=a_{u_i}-a_{u_{i+1}}+\cdots+(-1)^{p-i}a_{u_p}\ge1
\quad(0\le i\le p-1).
\]
Since $b_p=a_{u_p}\ge1$, it follows that $a_{u_i}=b_i+b_{i+1}\ge2$ for $0\le i<p$. The other two vertices of the triangle have positive exponents. This proves $\mathbf{a}(W)\ge2p+3$.

Since $G_L$ has only the two endpoint triangles, the strongly non-bipartite graph $G_L[U]$ has at most two connected components. If it is connected and contains both triangles, then $U=V(G_L)$, so $\deg(f)\ge L+5$. If it is connected and contains only one triangle, it is a triangle with an attached path of length $p=L$ or $L+1$. Hence $\deg(f)\ge2L+3\ge L+5$.

Finally, suppose that $G_L[U]$ has two components. Both endpoint triangles then belong to $U$, and the vertices omitted from the connecting path form a block of length $r\le2$ by domination. If the two attached paths have lengths $p$ and $q$, then $p+q=L-r$. Applying the preceding estimate to the two components yields
\[
\deg(f)\ge(2p+3)+(2q+3)=2L-2r+6\ge2L+2\ge L+5,
\]
where the last inequality holds for $L\ge3$.

Thus $\deg(f)\ge L+5$ in all cases. Lemma \ref{lem_deg_maximal_colon} gives $2s-1=\deg(f)$, so $s\ge\lceil(L+6)/2\rceil$. This proves the formula.
\end{proof}

\begin{theorem}\label{thm:counterexample}
For every integer $r\ge1$, there exists a graph $G$ such that
\[
\iota_{\rm ord}(G)-\iota_{\rm int}(G)=r.
\]
\end{theorem}

\begin{proof}
Take $L=2r+6$ and let $G=G_L$. By Lemma \ref{lem:sGL}, $\iota_{\rm ord}(G)=L$, while Lemma \ref{lem:int_GL} gives $\iota_{\rm int}(G)=(L+6)/2$. Hence
\[
\iota_{\rm ord}(G)-\iota_{\rm int}(G)=L-\frac{L+6}{2}=\frac L2-3=r
\]
as claimed.
\end{proof}

\begin{remark}
Mart\'inez-Bernal, Morey, and Villarreal showed that the ordinary and integral closure associated prime filtrations have the same eventual stable set, while their Example 3.9 shows that the two stability indices may differ \cite{MMV}. Theorem \ref{thm:counterexample} gives a stronger first appearance phenomenon: the maximal ideal may occur in an integral closure arbitrarily earlier than in the corresponding ordinary powers.
\end{remark}

The reverse difference can also be made arbitrarily large. We first compute both indices for disjoint unions of triangles.

\begin{lemma}\label{lem:triangles_indices}
Let $H=qK_3$ be the disjoint union of $q\ge1$ triangles. Then
\[
\iota_{\rm ord}(H)=q+1
\quad\text{and}\quad
\iota_{\rm int}(H)=
\begin{cases}
(3q+1)/2 & \text{if }q\text{ is odd},\\
(3q+4)/2 & \text{if }q\text{ is even}.
\end{cases}
\]
\end{lemma}

\begin{proof}
Let $F$ be the product of all variables of $H$. Each triangle contributes one edge factor to $F$, whereas multiplying $F$ by any variable allows two edge factors in the corresponding triangle. Thus $(I(H)^{q+1}:F)=\mathfrak{m}$ and $\iota_{\rm ord}(H)\le q+1$.

Conversely, suppose that $(I(H)^s:x^{\mathbf{a}})=\mathfrak{m}$. The matching capacity of $\mathbf{a}$ is $s-1$, and it is the sum of the matching capacities on the $q$ triangles. Each triangle must have positive matching capacity: if its support contained at most one vertex, multiplying by a variable already in that support (or by any variable if the support is empty) would not increase its matching capacity, contradicting the colon condition. Hence $s-1\ge q$, and $\iota_{\rm ord}(H)=q+1$.

Now suppose that $(\overline{I(H)^s}:x^{\mathbf{a}})=\mathfrak{m}$. For a triangle with vertices $u,v,w$, Lemma \ref{lem_capacity} gives
\[
a_v+a_w\ge a_u+1,\quad
a_u+a_w\ge a_v+1,\quad
a_u+a_v\ge a_w+1.
\]
These inequalities force all three exponents to be positive. Indeed, if $a_u=0$, the last two inequalities would give simultaneously $a_w\ge a_v+1$ and $a_v\ge a_w+1$. Thus $\deg(x^{\mathbf{a}})\ge3q$. By Lemma \ref{lem_deg_maximal_colon}, this degree is odd. When $q$ is odd, it follows that $2s-1\ge3q$. When $q$ is even, no triangle can have total weight four: up to permutation, the only positive integer triple with sum four is $(2,1,1)$, which violates one of the displayed inequalities. Consequently, the smallest possible odd total degree is $3q+3$.

These lower bounds are attained. If $q$ is odd, take all exponents equal to one. If $q$ is even, take exponents $(2,2,2)$ on one triangle and $(1,1,1)$ on all the others. In each case the inequalities of Lemma \ref{lem_capacity} hold on every triangle, hence on their disjoint union. The stated values of $\iota_{\rm int}(H)$ follow.
\end{proof}

\begin{theorem}\label{thm:reverse_counterexample}
For every integer $r\ge1$, there exists a graph $G$ such that
\[
\iota_{\rm int}(G)-\iota_{\rm ord}(G)=r.
\]
\end{theorem}

\begin{proof}
Take $G=(2r+1)K_3$. By Lemma \ref{lem:triangles_indices},
\[
\iota_{\rm int}(G)=3r+2
\quad\text{and}\quad
\iota_{\rm ord}(G)=2r+2.
\]
Their difference is $r$, as required.
\end{proof}

In particular, neither of the two associated prime containments holds for all graphs when $t\ge4$ is fixed. Indeed, Theorem \ref{thm:noncontainment} supplies one direction. For the other, take $G=(t-1)K_3$. Then $\iota_{\rm ord}(G)=t$, while Lemma \ref{lem:triangles_indices} gives $\iota_{\rm int}(G)>t$. Thus $\mathfrak{m}\in\operatorname{Ass}(S/I(G)^t)$ but $\mathfrak{m}\notin\operatorname{Ass}(S/\overline{I(G)^t})$.

\section{Comparison of v-numbers}

This section is devoted to the study of $v$-numbers of powers of an edge ideal and their integral closures. We start with a result for general squarefree monomial ideals.

\begin{lemma}
Let $I$ be a squarefree monomial ideal and let $P$ be a minimal prime of $I$. Then, for every $t\ge 1$,
\[
v_P(\overline{I^t})\le v_P(I^t).
\]
\end{lemma}
\begin{proof}
If $I=0$, the assertion is immediate. We may therefore assume that $I\ne0$. Let $f$ be a monomial such that $(I^t:f)=P$. Since $I$ is squarefree and $P$ is a minimal prime, $IS_P=PS_P$. Localizing the colon equality at $P$ gives
$
(P^tS_P:f)=PS_P,
$ 
which implies that $\deg_P(f)=t-1$.

Let $Q\in\Min(S/I)$ with $Q\ne P$. Choose a variable $x_i\in P\setminus Q$. Since $x_i f\in I^t\subseteq Q^t$ and $Q^t$ is $Q$-primary, we have $f\in Q^t$. It follows that
\[
(I^{(t)}:f)=\bigcap_{Q\in\Min(S/I)}(Q^t:f)=(P^t:f)=P.
\]
Since $I^t\subseteq\overline{I^t}\subseteq I^{(t)}$, we obtain
\[
P=(I^t:f)\subseteq(\overline{I^t}:f)\subseteq(I^{(t)}:f)=P.
\]
Thus $(\overline{I^t}:f)=P$, proving the assertion.
\end{proof}

When $I=I(G)$ is the edge ideal of a graph, we have the following stronger result.

\begin{theorem}\label{thm_local_v}
Let $G$ be a simple graph and let $I=I(G)$ be its edge ideal. Then, for every $t\ge 1$, every minimal prime $P\in\Min(S/I)$, and every monomial $f\in S$, we have
\[
(I^t:f)=P\quad\Longleftrightarrow\quad(\overline{I^t}:f)=P.
\]
In particular,
\[
v_P(\overline{I^t})=v_P(I^t).
\]
\end{theorem}

To establish Theorem \ref{thm_local_v}, we need the following lemma.

\begin{lemma}\label{lem_int_t}
Let $I=I(G)$ be the edge ideal of a simple graph $G$, and let $C$ be a vertex cover of $G$. Let $u=x^{\mathbf a}$ be a monomial of $S$ such that $\deg_C(u)=t$ and $u\in\overline{I^t}$. Then $u\in I^t$.
\end{lemma}
\begin{proof}
Since $u=x^{\mathbf a}\in\overline{I^t}$, there exist positive rational numbers $c_1,\ldots,c_k$ and minimal generators $x^{\mathbf b_1},\ldots,x^{\mathbf b_k}$ of $I^t$ such that $\sum_{i=1}^k c_i=1$ and $\mathbf a\ge\sum_{i=1}^k c_i\mathbf b_i$. Because $\deg_C(u)=t$ and $\deg_C(x^{\mathbf b_i})\ge t$ for every $i$, it follows that $\deg_C(x^{\mathbf b_i})=t$ for all $i$.

Thus each $x^{\mathbf b_i}$ is a product of $t$ edges having exactly one endpoint in $C$. Let $J$ be the edge ideal of the bipartite graph consisting of the edges of $G$ between $C$ and $V(G)\setminus C$. Then $x^{\mathbf b_i}\in J^t$ for all $i$, and hence $u\in\overline{J^t}$. Since $J$ is the edge ideal of a bipartite graph, it is normal (see \cite{SVV}), i.e., $\overline{J^t}=J^t$. Therefore $u\in J^t\subseteq I^t$.
\end{proof}

\begin{proof}[Proof of Theorem \ref{thm_local_v}]
If $I=0$, the assertion is immediate. We may therefore assume that $I\ne0$. Let $C$ be the minimal vertex cover corresponding to $P$. We first observe that if $f$ is a monomial such that $(J:f)=P$, where $J$ is either $I^t$ or $\overline{I^t}$, then
\[
\deg_C(f)=t-1.
\]
Indeed, $J\subseteq P^t$, as $P^t$ is integrally closed. Since $x_i f\in J$ for every $x_i\in P$, we have $\deg_C(f)\ge t-1$. Suppose that $\deg_C(f)\ge t$. Choose $t$ variables, counted with multiplicity, from $C$ whose product divides $f$. Because $C$ is a minimal vertex cover, each of these variables has a neighbor outside $C$. Let $g$ be the product of these neighbors, counted with multiplicity. Then $g\notin P$ and $gf\in I^t\subseteq J$, contradicting $(J:f)=P$. This proves the observation.

Now suppose that $(I^t:f)=P$. Since $I^t\subseteq\overline{I^t}\subseteq P^t$ and $\deg_C(f)=t-1$, we have
\[
P=(I^t:f)\subseteq(\overline{I^t}:f)\subseteq(P^t:f)=P.
\]
Thus $(\overline{I^t}:f)=P$.

Conversely, suppose that $(\overline{I^t}:f)=P$. For each variable $x_i\in P$, we have $x_i f\in\overline{I^t}$ and $\deg_C(x_i f)=t$. Lemma \ref{lem_int_t} gives $x_i f\in I^t$. Consequently,
\[
P\subseteq(I^t:f)\subseteq(\overline{I^t}:f)=P,
\]
so $(I^t:f)=P$. The equality of the local $v$-numbers follows.
\end{proof}

We now show that the difference $v(\overline{I(G)^t})-v(I(G)^t)$ can be arbitrarily large for any fixed integer $t\ge3$, while the difference $v(I(G)^t)-v(\overline{I(G)^t})$ can grow linearly in $t$.

\begin{theorem}
For any fixed integer $t\ge3$ and any positive integer $r$, there exists a graph $G$ such that
\[
v(\overline{I(G)^t})-v(I(G)^t)\ge r-1.
\]
\end{theorem}
\begin{proof}
Let $H_r$ be the graph consisting of a base triangle $abc$ together with $r$ triangles attached to each of its vertices. Define $G_{t,r}$ to be the disjoint union of $t-1$ copies of $H_r$, and set $I=I(G_{t,r})$.

Let $f$ be the product of all variables corresponding to the $t-1$ base triangles in $G_{t,r}$. Then $\deg(f)=3(t-1)$ and $f\notin I^t$, since each base triangle contributes at most one edge factor to $f$. However, $x_vf\in I^t$ for every variable $x_v$. Indeed, in the component containing $v$, the monomial $x_vf$ admits two edge factors, while the remaining $t-2$ base triangles each supply one edge factor. Thus $(I^t:f)=\mathfrak m$, giving
\[
v(I^t)\le3(t-1).
\]

We also have $f\in\overline{I^t}$. Indeed, assigning weight $1/2$ to each edge of every base triangle gives a fractional matching with capacities given by the exponent vector of $f$ and total weight $3(t-1)/2\ge t$. The assertion follows from Lemma \ref{lem_criterion_int}.

Let $g=x^{\mathbf b}$ be a monomial such that $(\overline{I^t}:g)=P$ is an associated prime of $\overline{I^t}$. Since $g\notin\overline{I^t}$ while $f\in\overline{I^t}$, we have $f\nmid g$. Thus at least one variable of a base triangle does not divide $g$.

Let $abc$ be a base triangle and suppose that $a\nmid g$. Consider any of the triangles $au_iv_i$ attached to $a$. If neither $u_i$ nor $v_i$ divides $g$, then the capacities of $a,u_i,v_i$ are all zero. Since $u_i$ and $v_i$ have no neighbors outside $\{a,u_i,v_i\}$, adding either $x_{u_i}$ or $x_{v_i}$ cannot increase the fractional matching number. By Lemma \ref{lem_criterion_int}, neither $x_{u_i}g$ nor $x_{v_i}g$ belongs to $\overline{I^t}$. Thus $x_{u_i},x_{v_i}\notin P$, contradicting $x_{u_i}x_{v_i}\in I\subseteq P$. Therefore at least one of $u_i,v_i$ divides $g$.

Let $k\ge1$ be the number of variables of the base triangles that do not divide $g$. Each such missing variable forces at least one variable from each of its $r$ attached triangles to divide $g$. These attached triangles have distinct vertices outside the base triangles. Consequently,
\[
\deg(g)\ge 3(t-1)-k+kr
=3(t-1)+k(r-1)\ge3(t-1)+r-1.
\]
It follows that $v(\overline{I^t})\ge3(t-1)+r-1$, and hence
\[
v(\overline{I^t})-v(I^t)\ge r-1
\]
as desired.
\end{proof}

\begin{theorem}\label{thm:ordinary-larger-arbitrarily}
For any fixed integer $t\ge4$ and any positive integer $r$, there exists a graph $G$ such that
\[
v(I(G)^t)-v(\overline{I(G)^t})\ge r.
\]
\end{theorem}

\begin{proof} For ease of reading, we divide the proof into several steps.

\smallskip
\noindent\textbf{Step 1: Graph Construction ($G_k$).}
\begin{itemize}
    \item \textbf{The Core Graph ($G$):} The graph $G$ consists of a triangle with vertices $\{a, b, c\}$, a $(2t-5)$-cycle with vertices $\{d_1, d_2, \dots, d_{2t-5}\}$, and an additional vertex $z$ adjacent to both $a$ and $d_1$. Let $V_{\text{core}} = \{a, b, c, d_1, \dots, d_{2t-5}, z\}$ denote the set of these core vertices.
    
    \item \textbf{The Attached Paths ($G_k$):} For each vertex $x \in V_{\text{core}}$, we attach $k$ paths of length $2$. That is, for each index $i \in \{1, \dots, k\}$, we introduce two new vertices, $x_{i1}$ and $x_{i2}$, along with the edges $\{x,x_{i1}\}$ and $\{x_{i1},x_{i2}\}$.
\end{itemize}

\smallskip
\noindent\textbf{Step 2: The $v$-number of the Integral Closure.}
By \cite[Theorem 2.10]{MV}, we have $\overline{I(G_k)^t} = I(G_k)^t + (abcd_1\cdots d_{2t-5})I(G)$. It then follows that
\[
\overline{I(G_k)^t} \colon (abcd_1\cdots d_{2t-5}z) = (x, x_{i1} \mid x \in V_{\text{core}} , \, i \in [k]).
\]
Hence, \( v(\overline{I(G_k)^t}) \le 2t - 1.
\) Moreover, as $v(\overline{I(G_k)^t}) \ge 2t - 1$, we conclude that
\[
v(\overline{I(G_k)^t}) = 2t - 1.
\]
\smallskip
\noindent\textbf{Step 3: Lower Bound Setup via the Set $\Omega$.} 
To establish a lower bound for $v(I(G_k)^t)$, let $F$ be a monomial such that $v(I(G_k)^t) = \deg F$ and $(I(G_k)^t \colon F)$ is a prime ideal. We define the subset of core vertices
\[
\Omega = \{x \in V_{\text{core}} \mid x \nmid F\}.
\]
For any vertex $x \in \Omega$ and for each $i \in [k]$, at least one of $x_{i1}$ or $x_{i2}$ must divide $F$. Indeed, if neither $x_{i1}$ nor $x_{i2}$ divides $F$, then $x_{i1}, x_{i2} \notin (I(G_k)^t \colon F)$, while $x_{i1}x_{i2} \in (I(G_k)^t \colon F)$, leading to a contradiction since $(I(G_k)^t \colon F)$ is a prime ideal.

Since each vertex in $\Omega$ forces at least $k$ distinct variables from the attached paths to divide $F$, while the remaining $2t - 1 - |\Omega|$ core vertices also divide $F$, we obtain the degree bound
\[
\deg F \ge k|\Omega| + (2t - 1 - |\Omega|) = (k - 1)|\Omega| + 2t - 1.
\]
Hence, if $\Omega \neq \emptyset$, then $|\Omega| \ge 1$, which yields
\[
v(I(G_k)^t) \ge (k - 1)|\Omega| + 2t - 1 \ge k - 1 + 2t - 1.
\]

\smallskip
\noindent\textbf{Step 4: The Case Where $\Omega = \emptyset$.} 
Suppose $\Omega = \emptyset$. Then every core vertex $x \in V_{\text{core}}$ divides $F$, so the core monomial 
\[
f = a b c d_1 \cdots d_{2t-5} z
\]
divides $F$. Observe that $f x \in I(G_k)^t$ for all $x \in V_{\text{core}}$, as well as for $x = x_{i1}$ with $x \in V_{\text{core}} \setminus \{z\}$ and $i \in [k]$. Thus, the remaining variables in $F/f$ can only come from $\{x_{i2} \mid x \in V_{\text{core}}, \, i \in [k]\}$ and $\{z_{i1} \mid i \in [k]\}$.

Moreover, for each $i \in [k]$, at least one of $z_{i1}$ or $z_{i2}$ must divide $F$. Indeed, if neither divides $F$, then $z_{i1}, z_{i2} \notin (I(G_k)^t \colon F)$ while $z_{i1}z_{i2} \in (I(G_k)^t \colon F)$, a contradiction. Hence,
\[
v(I(G_k)^t) = \deg F \ge 2t - 1 + k.
\]
Combining both cases, we deduce that $v(I(G_k)^t) \ge 2t - 1 + k - 1$. Subtracting $v(\overline{I(G_k)^t}) = 2t - 1$ gives
\[
v(I(G_k)^t) - v(\overline{I(G_k)^t}) \ge k - 1.
\]
By setting $k = r + 1$, we achieve the desired bound of $r$.
\end{proof}


\end{document}